\documentclass[12pt,reqno,intlim]{amsart}

\usepackage{amssymb,amsmath}

\usepackage[arrow,matrix,curve]{xy}

\newdir{ >}{{}*!/-5pt/\dir{>}}

\newcommand{\nc}{\newcommand}

\nc{\bC}{\bold{C}} \nc{\bN}{\Bbb{N}} \nc{\cF}{\mathcal{F}}
\nc{\cE}{\mathcal{E}} \nc{\cR}{\mathcal{R}} \nc{\cM}{\mathcal{M}}
\nc{\al}{\alpha} \nc{\bt}{\beta} \nc{\gm}{\gamma} \nc{\dl}{\delta}
\nc{\om}{\omega} \nc{\sg}{\sigma} \nc{\Sg}{\Sigma} \nc{\vf}{\varphi}
\nc{\ve}{\varepsilon} \nc{\os}{\overset} \nc{\ol}{\overline}
\nc{\ul}{\underline} \nc{\us}{\underset} \nc{\sbs}{\subset}
\nc{\bsl}{\backslash} \nc{\Ra}{\Rightarrow}
\nc{\lra}{\longrightarrow} \nc{\all}{\allowdisplaybreaks}

\nc{\Codes}{\operatorname{{\bold{Codes}}}}
\nc{\RegMono}{\operatorname{\mathcal{R}{\rm{eg}\mathcal{M}{\rm{ono}\!}}}}
\nc{\RegEpi}{\operatorname{\mathcal{R}{\rm{eg}\mathcal{E}{\rm{pi}\!}}}}
\nc{\Mn}{\operatorname{\mathcal{M}{\rm{ono}\!}}}
\nc{\Ep}{\operatorname{\mathcal{E}{\rm{pi}\!}}}
\nc{\Rg}{\operatorname{\mathcal{R}{\rm{eg}\!}}}
\nc{\Ob}{\operatorname{Ob\!}}

\numberwithin{equation}{section}

\newtheorem{theo}{\ \ \ Theorem}[section]
\newtheorem{lem}[theo]{\ \ \ Lemma}
\newtheorem{prop}[theo]{\ \ \ Proposition}
\newtheorem{cor}[theo]{\ \ \ Corollary}

\theoremstyle{definition}
\newtheorem{exmp}[theo]{\ \ \ Example}

\theoremstyle{remark}

\begin{document}

\title[]
{On the number of integer non-negative solutions of a linear Diophantine equation}

\author{Eteri Samsonadze}

\maketitle

\begin{abstract}
We deal with the problem to find the  number $P(b)$ of integer non-negative solutions of an equation $\sum_{i=1}^{n} a_i x_i=b$, where $a_1,a_2,...,a_n$ are natural numbers and $b$ is a non-negative integer. As different from the traditional methods of investigation of the function $P(b)$, in our study we do not employ the techniques of number series theory, but use in the main the properties of the Kronecker function and the elements of combinatorics. The formula is derived to express $P(b)$, for an integer non-negative $b$, via $P(r), P(r+M),...,P(r+(s-1)M)$ when $s\neq 0$, where  
$s=\left[ n-\dfrac{\sum_{i=1}^{n}a_i+r}{M} \right]$ and takes quite small values in some particular cases; $M$ is the least common multiple of the numbers $a_1,a_2,\ldots,a_n$, and $r$ is the remainder of $b$ modulo $M$. Also, the recurrent formulas are derived to calculate $P(b)$, for any non-negative integer $b$, which, in particular, are used in finding $P(r), P(r+M),...,P(r+(s-1)M)$. For the case where $s=0$ and $a_1,a_2,...,a_n$ are coprime, the explicit formula $P(b)=\dfrac{M^{n-1}}{a_1a_2\ldots a_n}C^{n-1}_{\begin{tiny}\left[\dfrac{b}{M}\right]\end{tiny}+n-1}$ is given.  To illustrate the proposed method, examples of finding the function $P(b)$ for linear Diophantine equations with $2,3,7$ and $n$ variables are given.\vskip+3mm


\end{abstract}

\section{Introduction}

In the present paper we study the problem of finding the number $P(b)$ of integer non-negative solutions of a linear Diophantine equation $$\sum_{i=1}^{n}a_ix_i=b,$$
\noindent where $a_1,a_2,...,a_n$ are natural numbers and $b$ is a non-negative integer. To this end we first study the interrelation of the numbers of integer non-negative solutions of arbitrary linear Diophantine equations with natural coefficients and one and the same number of variables. Applying the found interrelation, for the case  where the number $s$ defined by  
$$s=\left[ n-\dfrac{\sum_{i=1}^{n}a_i+r}{M} \right]$$
\noindent is different from $0$ we obtain the formula representing the number $P(b)$ via the numbers
 \begin{equation}
P(r), P(r+M),\ldots, P(r+(s-1)M),\end{equation} 
\noindent \noindent  where $M$ is the least common multiple of the numbers $a_1,a_2,\ldots,a_n$, and $r$ is the remainder of $b$ modulo $M$. It should be noted that applying a different method the formula representing  $P(b)$ via the numbers
$$P(r+M), P(r+2M), \ldots, P(r+nM)$$
 
\noindent is given in Riordan's book \cite{R}. Note also that $s\leq n-1$. Moreover, $s$ takes quite small values in some particular cases. 

In the paper we also give the recurrent formulas to calculate $P(b)$, for any non-negative integer $b$, which, in particular, are applied to find values (1.1).  

For the case where $s=0$ (which holds if, for instance, $r>(n-1)M-\sum_{i=1}^{n}a_i$) and the numbers $a_1,a_2,...,a_n$ are coprime (which can always be assumed without loss of generality), we give the explicit formula for $P(b)$:
 $$P(b)=\dfrac{M^{n-1}}{a_1a_2\ldots a_n}C^{n-1}_{\begin{tiny}\left[\dfrac{b}{M}\right]\end{tiny}+n-1}.$$



As different from the traditional methods of investigation of the function $P(b)$, in our study we do not employ the techniques of number series theory, but use in the main the properties of the Kronecker function and the elements of combinatorics. Moreover, the proposed method of finding the function $P(b)$ can be used even in the case where the numbers $a_1,a_2,...,a_n$ are not pairwise coprime. 

To illustrate the proposed method some examples of finding the number $P(b)$ in the case of linear Diophantine equations with $2,3,4,7$ and $n$ variables are given.



\section{Interrelationship of the numbers of integer non-negative solutions of some linear Diophantine equations}

In this section we consider the problem of determining the relationship between the numbers of integer non-negative 
solutions of linear Diophantine equations with one and the same number of variables. 

We denote by the symbol $P(f(x_1,x_2,...,x_n)=m)$ the number of integer non-negative solutions of an equation $f(x_1,x_2,...,x_n)=m$.

First, we represent the number of integer non-negative solutions  $P'(b)=P'(\sum_{i=1}^{n} a_i x_i=b)$ of the system 

$$\sum_{i=1}^{n} a_i x_i=b,$$
$$0\leq x_i \leq d_i-1 ~ (i=1,2,\cdots,n)$$

\noindent  via the values of the function $P(b)=P(\sum_{i=1}^{n} a_i x_i=b)$, where $n\geq 2$, $b$ is an integer non-negative number, $a_i\in \mathbb{N}$, $d_i\in \mathbb{N}$  $(i=1,2,\cdots,n)$.

\begin{lem} We have
\begin{equation}
P'(b)=P(b)-\sum_{i=1}^{n}P(b-a_id_i)+\sum_{1\leqslant i<j\leqslant n}^{n}P(b-a_id_i-a_jd_j)-$$
$$-\sum_{1\leqslant i<j<m\leqslant n}P(b-a_id_i-a_jd_j-a_m d_m)+\ldots +(-1)^{n}P(b-\sum_{i=1}^{n}a_id_i),
\end{equation}


\end{lem}

\begin{proof}
It is not difficult to observe that
\begin{equation} 
P(\sum_{i=1}^{n} a_i x_i=b)=\sum_{x_1}\sum_{x_2}...\sum_{x_n}\delta(\sum_{i=1}^{n} a_i x_i;b)
\end{equation}
\noindent and
\begin{equation}
P'(\sum_{i=1}^{n} a_i x_i=b)=\sum_{x_1=0}^{d_1-1}\sum_{x_2=0}^{d_2-1}...\sum_{x_n=0}^{d_n-1}\delta(\sum_{i=1}^{n} a_i x_i;b),
\end{equation}
\noindent where $\delta(x;y)$ is the Kronecker symbol, while $\sum_{x_1}\sum_{x_2}...\sum_{x_n}$ denotes the summation over all integer non-negative $x_1,x_2,...,x_n$.

Since $$\sum_{x}\delta(ax;b)=\sum_{x=0}^{d-1}\delta(ax;b)+\sum_{x=d}^{\infty}\delta(ax;b)=$$
$$\sum_{x=0}^{d-1}\delta(ax;b)+\sum_{x}\delta(a(x+d);b)=$$
$$\sum_{x=0}^{d-1}\delta(ax;b)+\sum_{x}\delta(ax;b-ad),$$
\noindent for any natural $d$, we conclude that
$$\sum_{x=0}^{d-1}\delta(ax;b)=\sum_{x}\delta(ax;b)-\sum_{x}\delta(ax;b-ad).$$

Applying this and the principle of mathematical induction, it is not difficult to show that
$$\sum_{x_1=0}^{d_1-1}\sum_{x_2=0}^{d_2-1}...\sum_{x_n=0}^{d_n-1}\delta(\sum_{i=1}^{n} a_i x_i;b)=$$
$$\sum_{x_1}\sum_{x_2}...\sum_{x_n}(\delta(\sum_{i=1}^{n} a_i x_i;b)-
\Sigma_{k=1}^{n}\delta(\sum_{i=1}^{n}a_ix_i;b-a_kd_k)+$$
$$\sum_{1\leq k<j\leq n}\delta(\sum_{i=1}^{n}a_ix_i;b-a_kd_k-a_jd_j)-$$
$$\sum_{1\leq k <j<m\leq n}\delta(\sum_{i=1}^{n}a_ix_i;b-a_kd_k-a_jd_j-a_md_m)+...+$$
$$(-1)^{n}\delta(\sum_{i=1}^{n}a_ix_i;b-\sum_{i=1}^{n}a_id_i)),$$
\noindent for any natural $d_1,d_2,...,d_n$. Therefore (2.2) and (2.3) imply that 
$$P'(b)=\sum_{\beta_1=0}^{1}\sum_{\beta_2=0}^{1}...\sum_{\beta_n=0}^{1}(-1)^{\beta_{1}+\beta_{2}+\ldots \beta_{n}}P(b-\sum_{i=1}^{n}\beta_ia_i d_i).$$



\end{proof}

\begin{theo}
The following equality holds:
\begin{equation}
P(\sum_{i=1}^{n} a_i x_i=b)=\sum_{k=0}^{\overline{s}}m_k P(\sum_{i=1}^{n} c_i x_i=\left[  \dfrac{b}{M} \right] - k   ),
\end{equation}
\noindent where $n\geqslant 2$; $a_i, c_i\in \mathbb{N}$ $(i=1,2,...,n)$; $b$ is a non-negative integer, $M$ is the least common multiple of $a_1,a_2,...a_n$; $r$ is the remainder of $b$ modulo $M$,
\begin{equation}
\overline{s}=\left[\sum_{i=1}^{n} c_i-\dfrac{\sum_{i=1}^{n} a_i +r}{ M} \right],
\end{equation} 
$$m_0=P(r),m_k=P(r+kM)-\Sigma_{i=1}^{n} P(r+kM-Mc_{i})+$$
$$\sum_{1\leq i<j\leq n} P(r+kM-M(c_{i}+c_{j}))-\ldots +$$
\begin{equation}
(-1)^{k}\sum_{1\leq i_1<i_2<\ldots <i_k\leq n}P(r+kM-M(c_{i_1}+c_{i_2}+\ldots +c_{i_k}))
\end{equation}
\noindent ($k=1,2,\ldots ,\overline{s}$),
\noindent and 
$$P(l)=P(\sum_{i=1}^{n}a_ix_i=l).$$
\end{theo}

\begin{proof} 

It is not difficult to observe that $$\sum_x\delta(ax;b)=\sum_x \sum_{t=0}^{d-1}\delta(a(dx+t);b),$$\noindent for any natural $d$. Hence we have 
$$\sum_{x_{1}} \sum_{x_{2}} ...\sum_{x_{n}}\delta(\sum_{i=1}^{n} a_i x_i;b)=$$
\begin{equation}
\sum_{x_{1}} \sum_{x_{2}} ...\sum_{x_{n}} \sum_{t_{1}=0}^{d_1-1} \sum_{t_{2}=0}^{d_2-1} ...\sum_{t_{n}=0}^{d_n-1}\delta(\sum_{i=1}^{n} a_i(d_i x_i+t_i);b),
\end{equation}
\noindent for any natural $d_1,d_2,...,d_n$. Therefore (2.2) and (2.3) imply that
$$P(\sum_{i=1}^{n} a_i x_i=b)=$$
\begin{equation}
\sum_{t_{1}=0}^{d_1-1} \sum^{d_2-1}_{t_{2}=0} ...\sum^{d_n-1}_{t_{n}=0}P(\sum_{i=1}^{n}d_ia_ix_i=b-\sum_{i=1}^{n}a_it_i),
\end{equation}
\noindent for any natural $d_i$ $(i=1,2,...,n)$. Applying this equality for $d_i=\dfrac{M} {a_i}c_i$ $(i=1,2,...,n)$, we obtain

$$P(\sum_{i=1}^{n} a_i x_i=b)=$$
\begin{equation}
 \sum_{t_{1}=0}^{\alpha_1}\sum^{\alpha_2}_{t_2=0} ...\sum^{\alpha_n}_{t_n=0}P(\sum_{i=1}^{n} c_ix_i=\dfrac{b-\sum^{n}_{i=1}a_it_i}{M}).
\end{equation}
\noindent where $\alpha_i=\dfrac{Mc_i}{a_i}-1$ $(i=1,2,...,n)$.

Since $\dfrac{b-\sum^{n}_{i=1}a_it_i}{M}$ is an integer if and only if $\sum^{n}_{i=1}a_it_i=r+Mk$, for some integer $k$, and using $0\leq\sum^{n}_{i=1}a_it_i\leqslant M\sum^{n}_{i=1}c_i-\sum^{n}_{i=1}a_i$, we have $$0\leq k\leqslant \sum^{n}_{i=1}c_i-\dfrac{\sum^{n}_{i=1}a_i+r}{M}.$$
\noindent Further, since $\dfrac{b-(r+Mk)}{M}=\left[  \dfrac{b} {M} \right] -k$, (2.9) implies (2.4), where $m_k$ $(k=0,1,2,..., \overline{s})$ is the number of integer non-negative solutions of the system:
$$\sum_{i=1}^{n} a_i t_i=r+Mk,~~ 
0\leqslant t_i\leqslant \dfrac{M}{a_i}c_i-1 ~~(i=1, 2,...n).$$

Applying Lemma 2.1 for $d_i=\dfrac{M} {a_i}c_i$ $(i=1, 2,...n)$, we obtain
$$m_k=\sum_{\beta_1=0}^{1}\sum_{\beta_2=0}^{1}...\sum_{\beta_n=0}^{1}(-1)^{\beta_{1}+\beta_{2}+\ldots \beta_{n}}P(r+Mk-M\sum_{i=1}^{n}\beta_ic_i),$$
\noindent for any $k$ $(k=0,1,...,\overline{s})$.

Since $c_i\geq 1$ $(i=1, 2,...n)$ and $P(b)=0$ for $b<0$, we obtain (2.6).
\end{proof}

\section{On the function $P(b)=P(\sum_{i=1}^{n}a_ix_i=b)$}

Theorem 2.1 generalizes our earlier result \cite{S}, where the number $P(b)$ of integer non-negative solutions of an equation $\sum_{i=1}^{n} a_i x_i=b$ $(a_1, a_2,...,a_n\in \mathbb{N}$, $b$ is a non-negative integer) was represented via the numbers of integer non-negative solutions of certain linear Diophantine equations where all coefficients of variables are 1.

Applying (2.4), (2.5) and (2.6) with $c_i=1$ $(i=1,2,...,n)$ we obtain 
\begin{equation}
 P(\sum_{i=1}^{n}a_ix_i=b)=\sum_{k=0}^{s}l_kP(\sum_{i=1}^{n}x_i=\left[  \dfrac{b}{M}\right]-k),
\end{equation}

\noindent where 
\begin{equation}
s=\left[  n-\dfrac{\sum_{i=1}^{n}a_i+r}{M} \right],
\end{equation}
$$l_0=P(r), ~l_k=P(r+kM)-C_n^{1}P(r+(k-1)M)+C_n^{2}P(r+(k-2)M)+$$
\begin{equation}
...+(-1)^{k}C_n^{k}P(r), ~(k=1,2,..,s). 
\end{equation} .

Since, for any integer $d$, we have 
\begin{equation}
 P(\sum_{i=1}^{n}x_i=d)=\left\{
\begin{array}{llllllll}
C_{d+n-1}^{n-1}&if&d\geq 0,\\
0&otherwise,
\end{array}
\right.
\end{equation}
\noindent we obtain the following statement.

\begin{theo} We have the equality
\begin{equation}
P(\sum_{i=1}^{n}a_ix_i=b)=\sum_{k=0}^{s}l_k \overline{C}_{\left[  \dfrac{b}{M}\right] +n-1-k}^{n-1},
\end{equation}
\noindent where $n\geq 2$, $a_i\in \mathbb{N}$ $(i=1,2,...,n)$, $b$ is a non-negative integer, $M$ is the least common multiple of the numbers $a_1,a_2,...,a_n$, $r$ is a remainder of $b$ modulo $M$, $s$ and $l_k$ $(k=0,1,...,s)$ are respectively given by (3.2) and (3.3), while 
 
\begin{equation}
 \overline{C}_k^{m}=\left\{
\begin{array}{llllllll}
C_k^{m}&if& k\geq m,\\
0&otherwise,
\end{array}
\right.
\end{equation}
\end{theo}

\vskip+2mm

Formulas (3.5), (3.2) and (3.3) imply:
\begin{equation}
P(b)=c_0\left[  \dfrac{b}{M}\right] ^{n-1}+c_1\left[  \dfrac{b}{M}\right] ^{n-2}+\ldots +c_{n-2}\left[  \dfrac{b}{M}\right]+c_{n-1},
\end{equation}
\noindent where $c_0$, $c_1$,...,$c_{n-1}$ are numbers which do not depend on $\left[\dfrac{b}{M}\right]$. In this way
we obtain the known result from \cite{B} by which the function $P(b)=P(\sum_{i=1}^{n}a_ix_i=b)$ is a polynomial of degree $(n-1)$ with respect to $\left[  \dfrac{b}{M}\right] $.

Formula (3.7) implies that
\begin{equation}
c_{n-1}=P(r).
\end{equation}


\vskip+3mm

Without loss of generality one can assume that the coefficients $a_1,a_2,$ $\ldots,a_n$ in the equation $\sum_{i=1}^{n}a_ix_i=b$ are coprime. In that case, according to \cite{PS}, we have

\begin{equation}
\lim_{b\rightarrow \infty} \dfrac{P(b)}{b^{n-1}}=\dfrac{1}{a_1a_2\ldots a_n(n-1)!}.
\end{equation}

Since $\left[ \dfrac{b}{M}\right] =\dfrac{b-r}{M}$, equality (3.7) implies that

$$\lim_{b\rightarrow \infty} \dfrac{P(b)} {b^{n-1}}=\dfrac{c_0}{M^{n-1}}.$$

\noindent Hence from 
(3.9) we obtain the value of the leading coefficient of polynomial (3.7)
\begin{equation}
c_0=\dfrac{M^{n-1}}{ a_1a_2\ldots a_n (n-1)!},
\end{equation}
\noindent for the case of coprime $a_1,a_2,$ $\ldots,a_n$. 

Formulas (3.7), (3.8), (3.10) imply
\begin{prop}
If  $b\equiv 0$ $(mod$ $M)$ and $a_1,a_2,\ldots,a_n$ are coprime, then the function $P(b)=P(\sum_{i=1}^{n}a_ix_i=b)$ is a polynomial of degree $(n-1)$ with respect to $b$ with the leading coefficient $\dfrac{1}{ a_1a_2\ldots a_n (n-1)!}$ and the free coefficient $1$.
\end{prop}

Since (3.5) implies that $c_0=\dfrac{\sum_{i=0}^{n}l_i}{(n-1)!}$, from (3.10) we obtain that if $a_1,a_2,\ldots,a_n$ are coprime, we have

\begin{equation}
\sum_{i=0}^{s}l_i=\dfrac{M^{n-1}}{a_1a_2\ldots a_n}.
\end{equation}

This implies that for calculating $P(b)$ by formula (3.5), it suffices to know the values of $l_0,l_1,\ldots, l_{s-1}$.
\vskip+3mm

From (3.3) we obtain
$$\sum_{i=0}^{s}l_i=P(r+sM)+P(r+(s-1)M)(1-C_n^{1})+$$
$$P(r+(s-2)M)(1-C_n^{1}+C_n^{2})+...+$$\vskip+1mm
\begin{equation}
P(r)(1-C_n^{1}+C_n^{2}+...+(-1)^{s}C_n^{s}).
\end{equation}\vskip+2mm

As is well known (see, e.g. \cite{ESIa}),

$$ C_n^{0}-C_n^{1}+C_n^{2}+...+(-1)^{m}C_{n}^{m}=\left\{
\begin{array}{lll}
(-1)^{m}C_{n-1}^{m}& if & m\leq n-1,\\
0& otherwise.
\end{array}
\right.$$

Therefore (3.11) and (3.12) imply that if $a_1,a_2,\ldots a_n$ are coprime, then we have
$$P(r+sM)-C_{n-1}^{1}P(r+(s-1)M)+C_{n-1}^{2}P(r+(s-2)M)+$$
\begin{equation}
...+(-1)^{s}C_{n-1}^{s}P(r)=\dfrac{M^{n-1}}{a_1a_2\ldots a_n}.
\end{equation}\vskip+3mm

Since the function $P(b)$ is a polynomial of degree $(n-1)$ with respect to $\left[ \dfrac{b}{M} \right] $,  it can be found applying the interpolation formula provided that its values at $n$ different points are known. In this manner the formula which represents the value of $P(b)$ via the values of $$P(r+M), P(r+2M),\ldots P(r+nM).$$\noindent was obtained in \cite{R}.

However if $s\neq 0$, then (3.5), (3.2), and (3.11) enables us to represent $P(b)$ via $$P(r), P(r+M), P(r+2M),\ldots P(r+(s-1)M),$$ \noindent where $s\leqslant n-1$; at that the value of $s$ is quite small in some particular cases (for instance, if $\sum_{i=1}^{n}a_i>>M$).

 For instance, for the equation $$2x_1+3x_2+ 3x_3+3x_4+6x_5+6x_6+6x_7+6x_8+6x_9+6x_{10}=6005$$ \noindent we have $s=\left[10- \dfrac{47+5}{6}\right] =1$.

 But if $s=0$ and the numbers $a_1,a_2,...,a_n$ are coprime, then from (3.5) and (3.11) we obtain
 $$P(b)=l_0 C_{\left[  \dfrac{b}{M}\right] +n-1}^{n-1},$$
 \noindent where $l_0=\dfrac{M^{n-1}}{ a_1a_2\ldots a_n}$, and since according to (3.3), $l_0=P(r)$, 
we obtain the following statement.
\begin{theo}
If $\left[n- \dfrac{\sum_{i=1}^{n}a_i+r}{M} \right]=0$ and the numbers $a_1,a_2,...,a_n$ are coprime, then
\begin{equation} 
P(\sum_{i=1}^{n}a_ix_i=r)=\dfrac{M^{n-1}}{a_1a_2\ldots a_n}.
\end{equation}
\noindent and
\begin{equation} 
P(\sum_{i=1}^{n}a_ix_i=b)=\dfrac{M^{n-1}}{a_1a_2\ldots a_n}C^{n-1}_{
\left[\dfrac{b}{M}\right]+n-1}
\end{equation}
\end{theo}\vskip+8mm

Since $r> (n-1)M-\sum_{i=1}^{n} a_i$ implies that $s=\left[n- \dfrac{\sum_{i=1}^{n}a_i+r}{M} \right]=0$, from Theorem 3.3 we obtain the following assertion. 
\begin{cor}
If $r> (n-1)M-\sum_{i=1}^{n} a_i$ and the numbers $a_1,a_2,...,a_n$ are coprime, then the values of $P(\sum_{i=1}^{n}a_ix_i=b)$ and $P(\sum_{i=1}^{n}a_ix_i=r)$ can be found by formulas (3.15) and (3.14).
\end{cor}

Consider now the case where $n=2$. Formulas (3.7), (3.8) and (3.10) imply the Ehrhart's result \cite{E} by which we have that
\begin{equation}
P(a_1x_1+a_2x_2=b)=\left [\dfrac{b}{M}\right]+P(r)
\end{equation}
\noindent if $(a_1,a_2)=1$ and $r$ is the remainder of $b$ modulo $M$, and, moreover, that $P(r$) is equal to either $0$ or $1$. 

At that \cite{E}
\begin{equation}
 P(r)=\left\{
\begin{array}{llllllll}
0&if& r<a_1+a_2,&a_1\nmid r,&and&a_2\nmid r;\\
1&if&r=a_1+a_2,&or&a_1\mid r&or&a_2\mid r.
\end{array}
\right.
\end{equation}
Moreover, as follows from Corollary 3.4,
\begin{equation}
P(r)=1
\end{equation}
\noindent if $r>a_1a_2-(a_1+a_2)$.  

The equality
$$(a_1a_2-(a_1+a_2))-(a_1+a_2)=a_1(a_2-2)-2a_2$$
implies that if either $a_1\geqslant 4$ and $a_2\geqslant 4$, or $a_1=3$ and $a_2>6$, then $a_1+a_2\leqslant a_1a_2-(a_1+a_2)$. As is not difficult to verify, in all other cases $a_1a_2-(a_1+a_2)<a_1+a_2$, and if $r\in (a_1a_2-(a_1+a_2), a_1+a_2]$, then $P(r)=1$. Therefore (3.16), (3.17) and (3.18) imply


\begin{prop} 
Let $(a_1,a_2)=1$ and $1<a_1<a_2$. If either $a_1=2$, or $a_1=3$ and $a_2< 6$, then we have 
\begin{equation}
 P(a_1x_1+a_2x_2=b)=\left\{
\begin{array}{llllllll}
\left[\dfrac{b}{a_1a_2}\right]+1,&if&r>a_1a_2-(a_1+a_2);\\
&or&r=a_1+a_2;\\
&or&a_1\mid r;\\
&or &a_2\mid r;\\
\left[\dfrac{b}{a_1a_2}\right],&otherwise.
\end{array}
\right.
\end{equation}
\noindent where $r$ is the remainder of $b$ modulo $a_1a_2$.
\end{prop} 
\vskip+3mm

\section{Recurrent formulas for the number of integer non-negative solutions of a linear Diophantine equation}

Applying the formulas given in Sections 2-3, we obtain
\begin{theo} We have
\begin{equation}
\sum_{k=0}^{a_1+a_2+...+a_n-n}l'_kP(\sum_{i=1}^{n}a_ix_i=b-k)=C^{n-1}
_{b+n-1}
\end{equation}
\noindent where 
$$l'_0=1,~ l'_k=C_{n-1+k}^{n-1}-\sum_{i=1}^{n}\overline{C}_{n-1+k-a_i}^{n-1}+\sum_{1\leq i< j\leq n}\overline{C}_{n-1+k-a_i-a_j}^{n-1}-$$

$$\sum_{1\leq i<j<m\leq n}\overline{C}_{n-1+k-a_i-a_j-a_m}^{n-1}+...$$
 \begin{equation}
(-1)^{n} \overline{C}_{n-1+k-\sum_{i=1}^{n}a_{i}}^{n-1},
 \end{equation}
 \noindent $(k=1,2,...,\sum_{i=1}^{n}a_i-n)$ and $\overline{C}^{m}_k$ is given by (3.6).
\end{theo}

\begin{proof} 
Using formula (2.8) for $a_i=1$ $(i=1,2,\ldots , n)$, we obtain
$$P(\sum_{i=1}^{n}x_i=b)=\sum_{t_1=0}^{d_1-1}\sum_{t_2=0}^{d_2-1}\ldots \sum_{t_n=0}^{d_n-1}P(\sum_{i=1}^{n}d_ix_i=b-\sum_{i=1}^{n}t_i),$$
\noindent for any natural $d_1,d_2,...,d_n$. This combined with formula (3.4) implies
\begin{equation}
\sum_{t_1=0}^{a_1-1}\sum_{t_2=0}^{a_2-1}\ldots \sum_{t_n=0}^{a_n-1}P(\sum_{i=1}^{n}a_ix_i=b-\sum_{i=1}^{n}t_i)=C_{b+n-1}^{n-1},
\end{equation}

\noindent whence we obtain (4.1), where $l'_k$ $(k=0,1,...,\sum_{i=1}^{n}a_i-n)$ is the number of integer non-negative solutions of the system 
$$\sum_{i=1}^{n}t_i=k,$$
$$0\leqslant t_i\leqslant a_i-1, (i=1,2\ldots ,n).$$

Equality (2.1) implies that
$$l'_k=\sum_{\beta_1=0}^{1}\sum_{\beta_2=0}^{1}\ldots \sum_{\beta_n=0}^{1}(-1)^{\beta_1+\beta_2+\ldots +\beta_n}P(\sum_{i=1}^{n}x_i=k-\sum_{i=1}^{n}\beta_i a_i),$$
\noindent for all $k=0,1,...,s$. This together with formula (3.4) gives: 
$$l'_k=\sum_{\beta_1=0}^{1}\sum_{\beta_2=0}^{1}\ldots \sum_{\beta_n=0}^{1}(-1)^{\beta_1+\beta_2+\ldots +\beta_n}\overline{C}^{n-1}_{k-\sum_{i=1}^{n}\beta_ia_i+n-1}.$$
\end{proof} 

Formula (4.1) implies the recurrent formula for $P(b)=P(\sum_{i=}^{n}a_ix_i=b),$ for any non-negative integer $b$. Namely, we have
\begin{cor}
\begin{equation}
P(\sum_{i=1}^{n}a_ix_i=b)=C_{b+n-1}^{n-1}-\sum_{k=1}^{a_1+a_2+...+a_n-n}l'_kP(\sum_{i=1}^{n}a_ix_i=b-k),
\end{equation}
\noindent where $l'_k$ $(k=1,2,...,\Sigma_{i=1}^{n}a_i-n)$ can be found by formula (4.2).
\end{cor}
\vskip+3mm

In a similar manner one can prove the following recurrent formula for $P(b)$.
\begin{theo} We have
\begin{equation}
P(b)=\sum_{k=0}^{s^{*}}l^{*}_k P(\left[ \dfrac{b}{m}\right] -k),
\end{equation}

\noindent where $m$ is an arbitrary natural number,
$$s^{*}=\left[ \dfrac{(m-1)\sum_{i=1}^{n}a_i -r)} {m}\right],$$
\noindent $r$ is the remainder of $b$ modulo $m$, and $l^{*}_k$ $(k=0,1,2,\ldots ,s*)$ is the number of integer non-negative solutions of the system 
$$\sum_{i=1}^{n}a_it_i=r+km,$$
$$0\leqslant t_i\leqslant m-1, (i=1,2,\ldots ,n).$$
\end{theo}
\begin{proof}
Applying (2.8) for $d_i=m$ $(i=1,2,...,n)$ we obtain:
\begin{equation}
P(\sum_{i=1}^{n}a_ix_i=b)=\sum_{t_1=0}^{m-1}\sum_{t_2=0}^{m-1}\ldots \sum_{t_n=0}^{m-1}P(\sum_{i=1}^{n}a_ix_i=\dfrac{b-\sum_{i=1}^{n}a_it_i}{m}).
\end{equation}
Since $\dfrac{b-\sum_{i=1}^{n}t_i}{m}$ is an integer if and only if $\sum_{i=1}^{n}a_it_i=r+km$ for some integer $k$, we obtain that $$0\leqslant r+km\leq \sum_{i=1}^{n}a_i(m-1)$$ and $$0\leq k\leq\dfrac{(m-1)\sum_{i=1}^{n}a_i-r}{m}.$$
\noindent And since $$\dfrac{b-(r+mk)}{m}=\left[\dfrac{b}{m}\right]-k,$$ \noindent equality (4.6) implies (4.5).
\end{proof}\vskip+6mm

In particular, from (4.5) for $m=2$ we obtain  the following statement.
\begin{cor} We have
\begin{equation}
 P(b)=\left\{
\begin{array}{lllll}
\sum_{k=0}^{s_1}P^{*}(2k)P(\dfrac{b}{ 2} -k)& if & b & is & even,\\
\\
\sum_{k=0}^{s_2}P^{*}(2k+1)P(\left[ \dfrac{b}{ 2}\right]  -k) & if & b & is & odd,
\end{array}
\right.
\end{equation}
\noindent where $$s_1=\left[\dfrac{\sum_{i=1}^{n}a_i}{ 2}\right],~ s_2=\left[\dfrac{\sum_{i=1}^{n}a_i-1}{ 2}\right],$$ 
\noindent and $P^{*}(d)$ denotes the number of integer non-negative solutions of the system 
$$\sum_{i=1}^{n}a_it_i=d,$$ 
$$0\leqslant t_i\leqslant 1, (i=1,2,\ldots,n).$$ 
\end{cor}
\vskip+3mm

To calculate $P^{*}(d)$ $(d\in \mathbb{N})$ one can apply the following recurrent formula
\begin{equation}
P^{*}_{n}(d)=P^{*}_{n-1}(d)+P^{*}_{n-1}(d-a_n),
\end{equation}
\noindent where $P^{*}_n(d)=P^{*}(d)$ and $P^{*}_{n-1}(c)$ $(c\in \mathbb{Z})$ denotes the number of integer non-negative solutions of the system
$$\sum_{i=1}^{n-1}a_it_i=c,$$ 
$$0\leq t_i\leq 1, ~(i=1,2,\ldots,n-1),$$

\noindent and also take into account the fact that if $c>\sum_{i=1}^{n-1}a_i$ or $c<0$, then $P^{*}_{n-1}(c)=0$.

\section{Examples}
In this section, employing the method described above we give examples of finding the number of integer non-negative solutions of a linear Diophantine equation. 

Note that for a Diophantine equation $\sum_{i=1}^{3}a_ix_i=b$ with pairwise coprime coefficients $a_1,a_2,a_3$, the formula is given in \cite{E} which represents the number $P(b)$ via $P(r)$, where $r$ is the remainder of $b$ modulo $a_1a_2a_3$. In this paper we deal with Diophantine equations where coefficients are not necessarily pairwise coprime.

\begin{exmp}
Find the number $P(b)$ of integer non-negative solutions of the equation
$$2x_1+4x_2+5x_3=b,$$
\noindent where $b$ is an integer non-negative number.

To this end we apply formulas (3.5), (3.2) and (3.3). Here $n=3$, $\sum_{i=1}^{3}a_i=11$, $M=20$, 
$$s=\left[3-\dfrac{11+r}{20}\right] =\left\{
\begin{array}{llllllll}
2&if&r\leq 9,
\\
1&if&r\geq 10.
\end{array}
\right.$$
\noindent where $r$ is the remainder of $b$ modulo $20$.

Since the coefficients of the equation are coprime, according to formula (3.11), $\sum_{i=0}^{2}l_i=10$ (at that, if $r\geqslant 10$, then $l_2=0$). Therefore, 
\begin{equation}
P(b)=l_0C_{b'+2}^{2}+l_1\overline{C}_{b'+1}^{2}+l_2\overline{C}_{b'}^{2},
\end{equation}
\noindent where $b'=\left[\dfrac{b}{20}\right]$,
\begin{equation}
l_0=P(r), ~l_1=P(r+20)-3P(r), ~l_2=10-l_0-l_1.
\end{equation}

To find $P(r)$ and $P(r+20)$ we apply formulas (4.4) and (4.2). We have $\Sigma_{i=1}^{n}a_i-n=8$. Moreover, $\overline{C}_l^{3}=0$ for $l<3$. Therefore, for any non-negative integer $d$ we obtain:
$$P(d)=C_{d+2}^{2}-\sum_{k=1}^{8}l'_{k}P(d-k),$$
\noindent where 
$$l'_k=C_{k+2}^{2}-(\overline{C}_k^{2}+\overline{C}_{k-2}^{2}+\overline{C}_{k-3}^{2})+(\overline{C}_{k-4}^{2}+\overline{C}_{k-5}^{2}), ~(k=1,2,...,s).$$
We get
$$l'_1=3, l'_2=5, l'_3=7, l'_4=8,l'_5=7, l'_6=5, l'_7=3, l'_8=1,$$
$$P(1)=0, P(2)=1, P(3)=0, P(4)=2, P(5)=1, P(6)=2, P(7)=1,$$
$$ P(8)=3, P(9)=2, P(10)=4, P(11)=2, P(12)=5, P(13)=3,$$
$$ P(14)=6, P(15)=4, P(16)=7, P(17)=5, P(18)=8, P(19)=6,$$
$$P(20)=10, P(21)=7, P(22)=11, P(23)=8, P(24)=13, P(25)=10,$$
$$P(26)=14, P(27)=11, P(28)=16, P(29)=13.$$
\vskip+2mm

Substituting these values in formula (5.2), one can find, for each $r=0,1,...,19$, the corresponding values of $l_0, l_1$ and $l_2$. The found values are given by the table below.

\begin{table}[ht]
\centering
\begin{tabular}{|c|cccccccccccc|}
\hline $r$&0;5&1&2;7&3&4;9&6;11&8;13&10;15&12;17&14;19&16&18\\
\hline $l_0$&1&0&1&0&2&2&3&4&5&6&7&8\\
 $l_1$&7&7&8&8&7&8&7&6&5&4&3&2\\
 $l_2$&2&3&1&2&1&0&0&0&0&0&0&0\\
\hline
\end{tabular}
\end{table}

Substituting these values in formula (5.1), one can find the value of the function $P(b)$, for any integer non-negative $b$. 

For instance, $P(214)=616$, since $\left[\dfrac{214}{20}\right]=10$, $r=14$, $l_0=6$, $l_1=4$, $l_2=0$, and $P(214)=6C_{12}^{2}+4C_{11}^{2}$.
\vskip+2mm

\end{exmp}
\vskip+2mm
\begin{exmp}
Find the number $P(826)$ of integer non-negative solutions of the equation
$$2x_1+3x_2+4x_3+6x_4=826.$$

We apply formulas (3.5), (3.2), (3.3), and (3.11).

 Here $n=4$, $\sum_{i=1}^{4}a_i=15$, $M=12$, $b=826$, $\left[\dfrac{b}{M}\right]=68$, $r=10$, $s=\left[4-\dfrac{15+10}{12}\right]=1$, $l_0+l_1=12$. Hence
\begin{equation}
P(826)=l_0C_{71}^{3}+l_1C_{70}^{3},
\end{equation}
\noindent where 
$$l_0=P(10), ~l_1=12-P(10).$$
As it is not difficult to observe, $P(10)=7$. Therefore $l_0=7$, $l_1=5$. From (5.3) we obtain

$$P(826)=7C_{71}^{3}+5C_{70}^{3}=673785.$$






\end{exmp}\vskip+2mm
\begin{exmp}
 Find the number $P(b)$ of integer non-negative solutions of the equation
$$2x_1+2x_2+3x_3+3x_4+3x_5+6x_6+6x_7=b,$$
\noindent where $b$ is an integer non-negative number. 

To this end we apply formulas (3.5), (3.3), (3.2), and (3.11). Here $n=7$, $\sum_{i=1}^{7}a_i=25$, $M=6$, $s=\left[7-\dfrac{25+r}{6}\right]=2$, $r$ is the remainder of $b$ modulo $6$,  $\sum_{i=0}^{2}l_i=12$. Hence
\begin{equation}
P(b)=l_0C_{b'+6}^{6}+l_1\overline{C}_{b'+5}^{6}+l_2\overline{C}_{b'+4}^{6},
\end{equation}
 \noindent where $b'=\left[\dfrac{b}{6}\right]$ and
\begin{equation}
l_0=P(r),~ l_1=P(r+6)-7P(r), ~l_2=12-l_0-l_1.
\end{equation}

Since $r+6\leq 11$, for finding $l_0$, $l_1$ and $l_2$ it suffices to find $P(c)$ for all $c\leq 11$.

To calculate $P(r)$ and $P(r+6)$ we apply recurrent formulas (4.7) and (4.8). As a result we obtain:\vskip+2mm

$P^{*}(0)=1,P^{*}(1)=0,P^{*}(2)=2,P^{*}(3)=3,P^{*}(4)=1,P^{*}(5)=6,P^{*}(6)=5,P^{*}(7)=3,P^{*}(8)=10,P^{*}(9)=7,P^{*}(10)=5,P^{*}(11)=14.$\vskip+2mm
Therefore, taking the fact that $P(c)=0$  for $c<0$ into account, we obtain that for even $c$ we have
$$P(c)=P(c')+2P(c'-1)+P(c'-2)+5P(c'-3)+10P(c'-4)+5P(c'-1),$$
while for odd $c$ we have
$$P(c)=3P(c'-1)+6P(c'-2)+3P(c'-3)+7P(c'-4)+14P(c'-5),$$
\noindent where $c'=\left[\dfrac{c}{2}\right]$.

Since $P(0)=1$, we obtain: 
\vskip+2mm
$P(1)=0,P(2)=2,P(3)=3,P(4)=3,P(5)=6,P(6)=12,P(7)=9,P(8)=21,P(9)=28,P(10)=30,P(11)=47.$\vskip+2mm
Substituting these values in equality (5.4) we can find $l_0$, $l_1$, and $l_2$ for $r=0, 1, 2, 3, 4, 5$.

Next, substituting, for each $r=0,1,...,5$, the corresponding values of $l_0$, $l_1$ and $l_2$ in formula (5.4), we obtain:

$$P(b)=\left\{
\begin{array}{llllllll}
C_{b'+6}^6~+~5\overline{C}_{b'+5}^{6}+6\overline{C}_{b'+4}^{6}&if&r=0,
\\
9\overline{C}_{b'+5}^6+3\overline{C}_{b'+4}^{6}&if&r=1,
\\
2C_{b'+6}^6+7\overline{C}_{b'+5}^{6}+3\overline{C}_{b'+4}^{6}&if&r=2,
\\
3C_{b'+6}^6+7\overline{C}_{b'+5}^{6}+2\overline{C}_{b'+4}^{6}&if&r=3,
\\
3C_{b'+6}^6+9\overline{C}_{b'+5}^{6}&if&r=4,
\\
6C_{b'+6}^6+5\overline{C}_{b'+5}^{6}+\overline{C}_{b'+4}^{6}&if&r=5.
\end{array}
\right.$$

For instance, $$P(49)=9C_{13}^{6}+3C_{12}^{6}=18216.$$\vskip+2mm
\end{exmp}
\begin{exmp}
 Find the number $P(b)$ of integer non-negative solutions of the equation
$$x_1+2x_2+3x_3+4x_4+5x_5+6x_6+7x_7=101.$$

Apply the recurrent formulas (4.7) and (4.8). Since $\sum_{i=1}^{7}a_i=28$, we need to find $P^{*}(c)$ for all $c\leq 28$. Applying formula (4.8), we obtain:\vskip+2mm

$P^{*}(0)=1$, $P^{*}(1)=1$, $P^{*}(2)=1$, $P^{*}(3)=2$, $P^{*}(4)=2$, $P^{*}(5)=3$, $P^{*}(6)=4$, $P^{*}(7)=5$, $P^{*}(8)=5$, $P^{*}(9)=6$, $P^{*}(10)=7$, $P^{*}(11)=7$, $P^{*}(12)=8$, $P^{*}(13)=8$, $P^{*}(14)=8$,
$P^{*}(15)=8$, $P^{*}(16)=8$, $P^{*}(17)=7$, $P^{*}(18)=7$, $P^{*}(19)=6$, $P^{*}(20)=5$, $P^{*}(21)=5$, $P^{*}(22)=4$, $P^{*}(23)=3$, $P^{*}(24)=2$, $P^{*}(25)=2$, $P^{*}(26)=1$, $P^{*}(27)=1$, $P^{*}(28)=1$.\vskip+2mm

Therefore from (4.7) we can conclude that if $c$ is even, then
$$P(c)=P(c')+P(c'-1)+2P(c'-2)+4P(c'-3)+$$
$$5P(c'-4)+7P(c'-5)+8P(c'-6)+8P(c'-7)+8P(c'-8)+7P(c'-9)+$$
$$5P(c'-10)+4P(c'-11)+2P(c'-12)+P(c'-13)+P(c'-14).$$

\noindent But if $c$ is odd, then
$$P(c)=P(c')+2P(c'-1)+3P(c'-2)+5P(c'-3)+$$
$$6P(c'-4)+7P(c'-5)+8P(c'-6)+8P(c'-7)+7P(c'-8)+6P(c'-9)+$$
\begin{equation}
5P(c'-10)+3P(c'-11)+2P(c'-12)+P(c'-13),
\end{equation}
\noindent where $c'=\left[ \dfrac{c}{ 2}\right] $.

Hence
$$P(101)=P(50)+2P(49)+3P(48)+5P(47)+6P(46)+7P(45)+8P(44)+$$
$$8P(43)+7P(42)+6P(41)+5P(40)+3P(39)+2P(38)+P(37).$$\vskip+2mm

In a similar manner, representing $P(50)$, $P(49)$,..., $P(37)$ via $P(25)$, $P(24)$, ..., $P(0)$ with the aid of (5.6), we obtain that
$$P(101)=P(25)+6P(24)+20P(23)+47P(22)+91P(21)+$$
$$154P(20)+230P(19)+312P(18)+389P(17)+445P(16)
+471P(15)+$$
$$463P(14)+420P(13)+352P(12)+272P(11)+190P(10)+120P(9)+$$
\begin{equation}
67P(8)+31P(7)+12P(6)+3P(5).\end{equation}\vskip+2mm

Taking the fact that $P(0)=1$ into account, from formula (5.6) we obtain:\vskip+2mm

$P(1)=1$, $P(2)=2$, $P(3)=3$, $P(4)=5$, $P(5)=7$, $P(6)=11$, $P(7)=15$, $P(8)=21$, $P(9)=28$, $P(10)=38$, $P(11)=49$, $P(12)=65$, $P(13)=82$, $P(14)=105$, $P(15)=131$, $P(16)=164$, $P(17)=201$, $P(18)=248$, $P(19)=300$, $P(20)=364$, $P(21)=436$, $P(22)=522$, $P(23)=618$, $P(24)=733$, $P(25)=860$.\vskip+2mm

Substituting these values in formula (5.7), we obtain
$$P(101)=628998.$$
\end{exmp}
\vskip+2mm

\begin{exmp}
Find the number $P(9005)$ of 
integer non-negative solutions of the equation
$$x_1+5x_2+10x_3+10x_4+...+10x_n=9005,$$
\noindent where $n\geq 3$.

Here $\sum_{i=1}^{n}a_i=6+10(n-2)$, $M=10$, $M(n-1)-\sum_{i=1}^{n}a_i=10(n-1)-6-10(n-2)=4$, $r=5$ is the remainder of $9005$ modulo $10$. Since $r> M(n-1)-\sum_{i=1}^{n}a_i$, Corollary 3.4  implies
$$P(9005)=\dfrac{M^{n-1}}{a_1a_2\ldots a_n}C^{n-1}_{\left[\dfrac{9005}{10} \right]+(n-1)}$$
Since $\dfrac{M^{n-1}}{a_1a_2\ldots a_n}=2$, we obtain
$$P(9005)=2C_{900+n-1}^{n-1}.$$\vskip+2mm
For instance, for $n=4$ we obtain $P(9005)=2C_{903}^{3}=244623302$.
\end{exmp}\vskip+2mm

\begin{exmp}
Find the number $P(b)$ of integer non-negative solutions of the equation 
$$x_1+kx_2+kx_3+\ldots +kx_n=b,$$

\noindent where $k,n\in \mathbb{N}$, $n\geq 2$, $b$ is a non-negative integer. Here we have $\sum_{i=1}^{n}a_i=1+k(n-1)$, $M=k$, $s=0$. Therefore (3.14) implies
\begin{equation}
P(x_1+kx_2+kx_3+\ldots +kx_n=b)=C^{n-1}_{\left[\dfrac{b}{k} \right]+(n-1)}.
\end{equation}
\end{exmp}
\vskip+2mm

\begin{exmp} 
Find the number $P(b)$ of integer non-negative solutions of the equation 
\begin{equation}
ax_1+2x_2+2ax_3+2ax_4+\ldots +2ax_n=b,
\end{equation}
\noindent where $n\geq 2$, $a$ is a natural odd number which is different from $1$, and $b$ is a non-negative integer.


To investigate the function $P(b)$ we apply equalities (3.5), (3.2), (3.3), and (3.11). We have
$$s=\left[n-\dfrac{a+2+2a(n-2)+r}{2a}\right]=\left[2-\dfrac{a+2+r}{2a}\right],$$ where $r$ is the remainder of $b$ modulo $2a$.

This implies that if $r>a-2$, we have $s=0$, and according to Theorem 3.3, we obtain that
\begin{equation} 
P(b)=C^{n-1}_{\left[ \dfrac{b}{2a}\right] +n-1}.
\end{equation}
\noindent in that case.

If $r\leqslant a-2$, then $s=1$. Moreover, formula (3.11) implies that $l_0+l_1=1$. Therefore we can conclude that either $l_0=1$ and $l_1=0$, or $l_0=0$ and $l_1=1$. Since $l_0=P(r)$, from (3.5) we obtain
that if $r\leqslant a-2$, then
\begin{equation}
 P(b)=\left\{
\begin{array}{lll}
C_{\left[ \dfrac{b}{2a}\right] +n-1}^{n-1}& if & P(r)=1,\\
\\
\overline{C}_{\left[ \dfrac{b}{2a}\right] +n-2}^{n-1} & if & P(r)=0.
\end{array}
\right.
\end{equation}
\vskip+2mm

Since $P(r)=P(ax_1+2x_2=r)$, formulas (3.19), (3.16), (3.17), (3.18) and (5.11) imply that

\vskip+4mm


$$P(b)=\left\{
\begin{array}{llllllll}
C^{n-1}_{\left[\dfrac{b}{2a}\right]+n-1}&if&r>a-2,\\&or&r&is&even,\\
\\
\overline{C}^{n-1}_{\left[\dfrac{b}{2a}\right]+n-2}&otherwise.
\end{array}
\right.$$\vskip+2mm
\end{exmp}
\vskip+15mm

\vskip+3mm

\textit{Author's address:}

\textit{Eteri Samsonadze,
Retd., I. Javakhishvili Tbilisi State University,}

\textit{1 Tchavchavadze Av., Tbilisi, 0179, Georgia}, 

\textit{e-mail: eteri.samsonadze@outlook.com} 


\begin{thebibliography}{99}


\bibitem{B} E. T. Bell, Interpolated denumerants and Lambert series, Amer. J. Math. 65(1943), 382-386.

\bibitem{E} E. Ehrhart, Sur un probl\'{e}me de g\'{e}ometrie diophantienne lin\'{e}are II. Journal fur reine und angewanandte Mathematik. 227(1967), 25-49.

\bibitem{ESIa} I. I. Ezhov, A. V. Skorokhod, I. I. Yadrenko, Elements of combinatorics, "Nauka", 1977.

\bibitem{PS} G. Polya, G. Szego, Aufgaben und Lehrs\H{a}tze aus der Analysis. Springer-Verlag, New York, 1964.

\bibitem{R} J. Riordan, An introduction to combinatorial analysis. John Willey  Sons, Inc., New York, 1958.

\bibitem{S}  E. T. Samsonadze, The formulas for the number of integer non-negative solutions of linear equations and inequalities. Proc. Tbilisi State University, 239 (1983), 36-42.


\end{thebibliography}
\end{document}